\documentclass[12pt]{amsart}

\usepackage{amssymb,amsbsy}
\usepackage{latexsym,array}
\usepackage{verbatim,color}
\usepackage{fullpage,euscript}

\usepackage[all]{xy}
\usepackage{xcolor}
\usepackage[colorlinks=true,linkcolor=blue,urlcolor=my_color,citecolor=magenta]{hyperref}

\definecolor{my_color}{rgb}{0,0.5,0.5}
\definecolor{MIXT}{rgb}{0.8,0.5,0.2}

\numberwithin{equation}{section}

\input {cyracc.def}
\font\tencyr=wncyr10 %scaled \magstephalf
\font\tencyi=wncyi10 %scaled \magstephalf
\font\tencysc=wncysc10 %scaled \magstephalf
\def\rus{\tencyr\cyracc}

\def\rusi{\tencyi\cyracc}
\def\rusc{\tencysc\cyracc}

\newcounter{rmke}%{\thermk}
\numberwithin{rmke}{section}

\newtheorem{thm}{Theorem}[section]
\newtheorem{lm}[thm]{Lemma}%[chapter]
\newtheorem{prop}[thm]{Proposition}%[chapter]

\theoremstyle{remark}
\newtheorem{rmk}[thm]{Remark}

\theoremstyle{definition}
\newtheorem{ex}[thm]{Example}

\theoremstyle{plain}

\newcommand {\g}{{\mathfrak g}}

\newcommand {\h}{{\mathfrak h}}

\newcommand {\el}{{\mathfrak l}}

\newcommand {\n}{{\mathfrak n}}

\newcommand {\q}{{\mathfrak q}}
\newcommand {\rr}{{\mathfrak r}}
\newcommand {\es}{{\mathfrak s}}

\newcommand{\gt}{\mathfrak}

\newcommand {\eus}{\EuScript}

\newcommand {\gS}{{\eus S }}

\newcommand {\tvp}{\tilde\varpi}

\newcommand {\BZ}{{\mathbb Z}}

\newcommand {\md}{/\!\!/}

\newcommand {\codim}{{\mathrm{codim\,}}}

\newcommand {\ind}{{\mathsf{ind\,}}}
\newcommand {\Lie}{{\mathsf{Lie\,}}}

\newcommand {\rk}{{\mathrm{rk\,}}}

\newcommand {\spe}{{\mathsf{Spec\,}}}

\newcommand {\trdeg}{{\mathrm{tr.deg\,}}}

\newcommand {\tri}{\mathfrak{sl}_2}
\newcommand {\GR}[2]{{\textrm{{\bf #1}}}_{#2}}

\newcommand {\ov}{\overline}
\newcommand {\un}{\underline}

\newcommand {\gig}{\mathsf{g.i.g.}}

\newcommand {\beq}{\begin{equation}}
\newcommand {\eeq}{\end{equation}}

\renewcommand{\le}{\leqslant}
\renewcommand{\ge}{\geqslant}

\newcommand{\odin}{\mathrm{{1}\!\! 1}}
\newcommand{\bH}{\boldsymbol{H}}

\newcommand {\bbk}{\Bbbk}%{\mathbb F}%

\begin{document}
\setlength{\parskip}{3pt plus 2pt minus 0pt}
\hfill { {\color{blue}\scriptsize April 30, 2017}}%%
\vskip1ex

\title[Symmetric invariants related to representations of exceptional groups]
{Symmetric invariants related to representations of exceptional simple groups}
\author[D.\,Panyushev]{Dmitri I. Panyushev}
\address[D.P.]%
{Institute for Information Transmission Problems of the R.A.S, Bolshoi Karetnyi per. 19, 
Moscow 127051, Russia}
\email{panyushev@iitp.ru}
%\urladdr{\url{http://www.mccme.ru/~panyush}}
\author[O.\,Yakimova]{Oksana S.~Yakimova}
\address[O.Y.]{Institut f\"ur Mathematik, Friedrich-Schiller-Universit\"at Jena,  07737 Jena, 
Deutschland}
\email{oksana.yakimova@uni-jena.de}
\thanks{The research of the first author was carried out at the IITP RAS at the expense of the Russian Foundation for Sciences (project {\rus N0} 14-50-00150).  The second author is partially supported by  Graduiertenkolleg GRK 1523 ``Quanten- und Gravitationsfelder".}
\keywords{index of Lie algebra, coadjoint representation, symmetric invariants}
\subjclass[2010]{14L30, 17B08, 17B20, 22E46}
\begin{abstract}
We classify the 
finite-dimensional rational representations $V$ of the exceptional algebraic groups $G$ with $\g=\Lie G$ 
such that the symmetric invariants of the semi-direct product $\g\ltimes V$, where $V$ is an Abelian ideal, form a polynomial ring.
\end{abstract}
\maketitle

\begin{flushright}
{\small{\it To our teacher Ernest B.~Vinberg on occasion of his 80-th birthday}}
 %with deepest respect, gratitude, and admiration}}
\end{flushright}

%%%%%%%%%%%%%%%%%%%%%%%%%%
\section*{Introduction}

\noindent
The ground field $\bbk$ is algebraically closed and of characteristic $0$. 
In 1976, Vinberg et al{.} classified the irreducible representations of simple algebraic groups with 
polynomial rings of invariants~\cite{KPV76}. Such representations are sometimes called {\it coregular}.
The most important class of coregular representations of reductive groups is provided by the $\theta$-groups introduced and studied in depth by Vinberg, see~\cite{V76}.  
Since then 
the classification of coregular representations of semisimple groups has attracted much attention. The 
{\bf reducible} coregular representations of {\bf simple} groups have been classified independently by 
Schwarz~\cite{gerry1} and Adamovich--Golovina~\cite{ag79}, while the {\bf irreducible} coregular representations of {\bf semisimple} groups are classified by Littelmann~\cite{litt}.

For several decades, 
only rings of invariants of representations of {\bf reductive} groups were considered. However, invariants   
of non-reductive groups are also very important in Representation Theory.
Let $S$ be an algebraic group 
with $\es=\Lie S$. The invariants of $S$ in the symmetric algebra $\gS (\es)$ of $\es$ 
(=\,{\it symmetric invariants\/} of $\es$ or $S$) help us to understand 
the coadjoint action $(S:\es^*)$ and in particular, coadjoint orbits, as well as representation theory of $S$.
Several classes of non-reductive groups $S$ such that $\gS (\es)^S$ is a polynomial ring have been
found recently, see e.g. \cite{J,trio,coadj,Y}. A quest for this type of groups continues. 
Hopefully, one can find interesting properties of $S$ and its representations under the assumption that 
the ring $\gS (\es)^S$ is polynomial. 
 
A natural class of non-reductive groups, which is still tractable, is given by a semi-direct product 
construction, see Section~\ref{sect:kos-th} for details. In~\cite{Y}, the following problem has been proposed:
to classify all representations $V$ of simple algebraic groups $G$ such that 
the ring of symmetric invariants of the semi-direct product $\q=\g\ltimes V$ is polynomial (in other words, 
the coadjoint representation of $\q$ is coregular). It is easily seen that if $\q$ has this property,
then $\bbk[V^*]^G$ is a polynomial ring, too. Therefore, the suitable representations $(G,V)$ have to
be extracted from the lists of \cite{gerry1,ag79}.
Some natural representations of $G=SL_n$ are studied in~\cite{Y}. Those considerations imply that the 
$SL_n$-case is very difficult. For this reason, we take here the other end and classify such representations
$(G,V)$ for the {\it exceptional} algebraic groups $G$. In a forthcoming article, we provide such a 
classification for the representations of the orthogonal and symplectic groups. To a great extent, our classification results rely on the theory developed by the second author in \cite{Y16}.

{\sl \un{ Notation}.}
Let $S$ act on an irreducible affine variety $X$. Then $\bbk[X]^S$ 
is the algebra of $S$-invariant regular functions on $X$ and $\bbk(X)^S$
is the field of $S$-invariant rational functions. If $\bbk[X]^S$ is finitely generated, then 
$X\md S:=\spe \bbk[X]^S$.  Whenever $\bbk[X]^S$ is a graded polynomial ring, 
the elements of any set of algebraically independent homogeneous generators 
will be referred to as {\it basic invariants\/}. If $V$ is an $S$-module and $v\in V$, then 
$\es_v=\{\zeta\in\es\mid \zeta{\cdot}v=0\}$ is the {\it stabiliser\/} of $v$ in $\es$ and 
$S_v=\{s\in S\mid s{\cdot}v=v\}$ is the {\it isotropy group\/} of $v$ in $S$.

In explicit examples of Section~\ref{sect:rank} and in Table~\ref{table-ex1}, 
we identify the representations $V$ of semisimple groups with their highest weights, using the {\it multiplicative\/} notation and the 
Vinberg--Onishchik numbering of the fundamental weights~\cite{VO}. For instance, if $\varpi_1,\dots,\varpi_n$ are the fundamental weights of a simple algebraic group $G$, then $V=\varpi_i^2+2\varpi_{j}$ stands for the direct sum of three simple $G$-modules, with highest weights $2\varpi_i$ (once) and $\varpi_{j}$ (twice). If $H\subset G$ is semisimple and we are describing the restriction of $V$ to $H$ (i.e., $V\vert_H$), then the fundamental weights of 
$H$ are denoted by $\tvp_i$. Write `$\odin$' for the trivial one-dimensional representation.

%%%%%%%%%%%%%%%%%%%%%%%%%%%%%%%%%%%%%%%%%%%%%
\section{Preliminaries on the coadjoint representations}
\label{sect:prelim}

\noindent
Let $S$\/ be an affine algebraic group with Lie algebra $\es$. The symmetric algebra 
$\gS (\es)$ over $\bbk$ is identified with the graded algebra of polynomial functions on $\es^*$ and we 
also write $\bbk[\es^*]$ for it.  
The {\it index of}\/ $\es$, $\ind\es$, is the minimal codimension of $S$-orbits in $\es^*$. Equivalently,
$\ind\es=\min_{\xi\in\q^*} \dim \es_\xi$. By Rosenlicht's theorem~\cite[2.3]{VP}, one also has
$\ind\es=\trdeg\bbk(\es^*)^S$. The ``magic number'' associated with $\es$ is $b(\es)=(\dim\es+\ind\es)/2$. Since the coadjoint orbits are even-dimensional, the magic number is an integer. If $\es$ is reductive, then  
$\ind\es=\rk\es$ and $b(\es)$ equals the dimension of a Borel subalgebra. The Poisson bracket $\{\ ,\ \}$ in
$\bbk[\es^*]$ is defined on the elements of degree $1$ (i.e., on $\es$) by $\{x,y\}:=[x,y]$. 
The {\it centre\/} of the Poisson algebra $\gS(\es)$ is $\gS(\es)^\es=\{F\in \gS(\es)\mid \{F,x\}=0 \ \ \forall x\in\es\}$. If $S^o$ is the identity component of $S$, then $\gS(\es)^\es=\gS(\es)^{S^o}$.

The set of $S$-{\it regular\/} elements of $\es^*$ is 
$\es^*_{\sf reg}=\{\eta\in\es^*\mid \dim S{\cdot}\eta\ge \dim S{\cdot}\eta' \text{ for all }\eta'\in\es^*\}$.
We say that $\es$ has the {\sl codim}--$2$ property if $\codim (\es^*\setminus\es^*_{\sf reg})\ge 2$.
The following useful result appears in~\cite[Theorem\,1.2]{coadj}:
\\[.7ex]
\hbox to \textwidth{\ $(\blacklozenge)$ \hfil
\parbox{400pt}{\it
Suppose that $S$ is connected, $\es$ has the {\sl codim}--$2$ property, and there are homogeneous algebraically
independent $f_1,\dots,f_l\in \bbk[\es^*]^S$ such that $l=\ind\es$ and $\sum_{i=1}^l \deg f_i=b(\es)$.
Then $\bbk[\es^*]^S=\bbk[f_1,\dots,f_l]$ and $(\textsl{d}f_1)_\xi,\dots,(\textsl{d}f_l)_\xi$ are linearly
independent if and only if $\xi\in\es^*_{\sf reg}$.}
\hfil}

\noindent
More generally, one can define the set of $S$-regular elements for any $S$-action on an irreducible variety $X$; that is, $X_{\sf reg}=\{x\in X\mid \dim S{\cdot}x\ge \dim S{\cdot}x' \text{ for all } x'\in X\}$.

We say that the action $(S:X)$ {has a generic stabiliser}, if there exists
a dense open subset $\Omega\subset X$ such that all stabilisers $\es_x$, $x\in \Omega$, are 
$S$-conjugate. Then {\bf any} subalgebra $\es_x$, $x\in\Omega$, is called a {\it generic stabiliser}
(=\,\textsf{g.s.}). The points of $\Omega$ are said to be $S$-generic (or, just generic if the group is clear from the context).
Likewise, one defines  a {\it generic isotropy group} (=\,$\gig$), 
which is a subgroup of $S$. By~\cite[\S\,4]{Ri}, $(S:X)$ has a generic stabiliser if and only if it has a generic isotropy group. It is also shown therein that $\gig$ always exists if $S$ is reductive and $X$ is smooth.
If $H$ is a generic isotropy group for $(S:X)$ and $\h=\Lie H$, then we write $H=\gig(S:X)$ and 
$\h=\textsf{g.s.}(S:X)$.
A systematic treatment of generic stabilisers in the context of reductive group
actions can be found in \cite[\S 7]{VP}. 

Note that if a generic stabiliser for $(S:X)$ exists, then any $S$-generic point is $S$-regular, but not vice versa.

Recall that $f\in \bbk[X]$ is called a {\it semi-invariant} of $S$ if 
$S{\cdot}f\subset \bbk f$. A semi-invariant $f$ is {\it proper} if 
$f\not\in \bbk[X]^S$.

%%%%%%%%%%%%%%%%%%%%%%%%%%%%%%%%%%%%%%%%%%%%%
\section{On the coadjoint representations of semi-direct products}
\label{sect:kos-th}

In this section, we gather some results on the coadjoint representation that are specific for 
semi-direct products. In particular, we recall the necessary theory from \cite{Y16}.
\\ \indent
Let $G\subset GL(V)$ be a connected algebraic group with $\g=\Lie G$.  
The vector space $\g\oplus V$ has a natural structure of Lie algebra, the {\it semi-direct product 
of\/ $\g$ and $V$}.
Explicitly, if $x,x'\in \g$ and $v,v'\in V$, then
\[
   [(x,v), (x',v')]=([x,x'], x{\cdot}v'-x'{\cdot}v) \ .
\]
This Lie algebra is denoted by $\es=\g\ltimes V$, and $V\simeq \{(0,v)\mid v\in V\}$ 
is an abelian ideal of $\es$.
The corresponding connected algebraic group $S$ is the semi-direct 
product of $G$ and the commutative unipotent group $\exp(V)\simeq V$. 
The group $S$ can be identified with  $G\times V$,  the product being given by
\[
    (s,v)(s',v')= (ss', (s')^{-1}{\cdot}v+v'), \ \text{ where } \ s,s'\in G .
\]
In particular,  $(s,v)^{-1}=(s^{-1}, -s{\cdot}v)$. Then $\exp(V)$ can be identified with
$1\ltimes V:=\{(1,v)\mid v\in V\} \subset G\ltimes V$.
If $G$ is reductive, then the subgroup $1\ltimes V$ is the unipotent radical of $S$, also denoted
by $R_u(S)$.

There is a general formula for the index of $\es=\g\ltimes V$, which is due to M.~Ra\"is \cite{rais}. Namely, there is a dense open subset $\Omega\subset V^*$ such that $\ind\es=\trdeg\bbk(V^*)^G+\ind\g_\xi$ for any $\xi\in\Omega$.
In particular, if a generic stabiliser for $(G:V^*)$ exists, then one can take $\g_\xi$ to be a generic stabiliser.

\begin{rmk}   \label{rem:useful}
There are some useful observations related to the symmetric invariants of the semi-direct product 
$\es=\g\ltimes V$:
\begin{itemize}
\item[\sf (i)] \ The decomposition $\es^*=\g^*\oplus V^*$ yields a bi-grading of 
$\bbk[\es^*]^S$~\cite[Theorem\,2.3(i)]{coadj}. If ${\bH}$ is a bi-homogenous $S$-invariant, then
$\deg_{\g}\! {\bH}$ and $\deg_{V}\! {\bH}$ stand for the corresponding degrees;
\item[\sf (ii)] \ The algebra $\bbk[V^*]^G$ is contained in $\bbk[\es^*]^S$. Moreover, a minimal generating
system for $\bbk[V^*]^G$ is a part of a minimal generating system of 
$\bbk[\es^*]^S$~\cite[Sect.\,2\,(A)]{coadj}. Therefore, if $\bbk[\es^*]^S$ is a polynomial ring, then so is 
$\bbk[V^*]^G$.
\end{itemize}
\end{rmk}

\begin{prop}[Prop.\,3.11 in \cite{Y16}]         \label{non-red} 
Let $G$ be a connected algebraic group acting on a finite-dimensional vector space $V$. Suppose that 
$G$ has no proper semi-invariants in\/ $\bbk[\es^*]^{1\ltimes V}$ and\/ $\bbk[\es^*]^S$ is a polynomial 
ring in $\ind \es$ variables. 
For  generic $\xi\in V^*$, we then have 
\begin{itemize}
\item the restriction map 
$\psi: \bbk[\es^*]^S \to \bbk[\g^*{\times}\{\xi\}]^{G_\xi{\ltimes}V}\simeq \gS(\g_\xi)^{G_\xi}$ is surjective; 
\item $\gS(\g_\xi)^{G_\xi}$ coincides with $\gS(\g_\xi)^{\g_\xi}$; 
\item $\gS(\g_\xi)^{G_\xi}$ is a polynomial ring in $\ind\g_\xi$ variables. 
\end{itemize}
\end{prop}
\noindent
Note that $G$ is not assumed to be reductive and $G_\xi$ is not assumed to be connected in the above proposition! We mention also that there are  
isomorphisms 
$\bbk[\g^*{\times}\{\xi\}]^{G_\xi{\ltimes}V}\simeq \bbk[\gt g_x]^{G_x}\simeq \gS(\g_\xi)^{G_\xi}$ for any $\xi\in V^*$, see \cite[Lemma~2.5]{Y16}.

From now on, $G$ is supposed to be reductive. The action $(G:V)$ is said to be {\it stable\/} if the union of closed $G$-orbits is dense in $V$. Then a generic stabiliser $\mathsf{g.s.}(G:V)$
is necessarily reductive.

Consider the following assumptions on $G$ and $V$:

\begin{itemize}
\item[($\diamondsuit$)] \  
the action of $(G:V^*)$ is stable, $\bbk[V^*]^G$ is a  polynomial ring, $\bbk[\g^*_\xi]^{G_\xi}$ is a polynomial 
ring for generic $\xi\in V^*$, and $G$ has no proper semi-invariants in $\bbk[V^*]$.
\end{itemize}

The following result of the second author was excluded from the final text of 
\cite{Y16}. .

\begin{thm}     \label{V-rank-1}
Suppose that $G$ and $V$ satisfy condition $(\diamondsuit)$ and $V^*\md G=\mathbb A^1$, i.e.,
$\bbk[V^*]^G=\bbk[F]$ for some homogeneous $F$. Let $H$ be a generic isotropy group for $(G:V^*)$
and $\h=\Lie H$.
Assume further that $D=\{x\in V^*\mid F(x)=0\}$ contains an open $G$-orbit, say $G{\cdot} y$, 
$\ind\g_y=\ind\h=:\ell$, and $\gS(\g_y)^{G_y}$ is a polynomial ring in $\ell$ variables  
with the same degrees of generators as $\gS(\h)^{H}$.  
Then $\bbk[\es^*]^S$ is a polynomial ring in $\ind\es=\ell+1$ variables.
\end{thm}
\begin{proof}
If $\ell=0$, then $\bbk[\es^*]^S=\bbk[F]$ and we have nothing to do. 
Assume that $\ell\ge 1$.
Let $\{{\bH}_i\mid 1\le i\le \ell\}$ be bi-homogeneous $S$-invariants chosen as  
in \cite[Lemma\,3.5(i)]{Y16}.  
Assume that $\deg_{\g} {\bH}_i\le \deg_{\g} {\bH}_j$ if $i<j$. We will show that these polynomials 
can be modified in such a way that the new set 
satisfies the conditions of  \cite[Lemma\,3.5(ii)]{Y16} and therefore freely generates 
$\bbk[\es^*]^S$ over $\bbk[F]$. 

Notice that $F$ is an irreducible polynomial, 
because $(\diamondsuit)$ includes also the absence of proper semi-invariants. 
Thereby $\bbk[D]^G=\bbk$ and a non-trivial relation over $\bbk[D]^G$ among 
$\tilde {\bH}_i={{\bH}_i}\vert_{\g^*{\times} D}$ 
gives also a non-trivial relation among  $\tilde{\bf h}_i={{\bH}_i}\vert_{\g^*{\times}\{y\}}$. Recall that $\tilde{\bf h}_i\in \gS(\gt g_y)^{G_y}$ 
by \cite[Lemma~2.5]{Y16}.
Assume that 
$\tilde{\bf h}_1,\ldots,\tilde{\bf h}_j$ are algebraically independent 
if $j=d$ and dependent for $j=d+1$. Then $\tilde{\bf h}_{d+1}$ is not among the generators of 
$\gS(\g_y)^{G_y}$ and it can be expressed as a polynomial 
$R(\tilde{\bf h}_1,\ldots,\tilde{\bf h}_d)$. Then also 
$\tilde {\bH}_{d+1}-R(\tilde {\bH}_1,\ldots,\tilde {\bH}_d)=0$ and we can replace 
${\bH}_{d+1}$ by the bi-homogeneous part of $\frac{1}{F}({\bH}_{d+1}-R({\bH}_1,\ldots,{\bH}_d))$ 
of bi-degree $(\deg_{\g} {\bH}_{d+1},\deg_{V} {\bH}_{d+1}-\deg F)$.
Clearly, $\sum_i\deg {\bH}_i$ is decreasing and therefore the process will end up at some stage and bring 
a new set $\{{\bH}_i\}$ satisfying the conditions of \cite[Lemma\,3.5(ii)]{Y16}. 
\end{proof}

\begin{rmk}    \label{rem:deg-V}
Although Theorem~\ref{V-rank-1} asserts that $\bbk[\es^*]^S$ is a polynomial ring under certain conditions, 
it does not say anything about the (bi)degrees of the generators ${\bH}_i$. Finding these degrees is not an 
easy task. Let us say a few words about it. 

Suppose that $G$ is semisimple, $\bbk[\gt s^*]^S$ is a polynomial ring, and 
$\gt s$ has the {\sl codim}--$2$ property. As is mentioned above, 
$\ind\es={\rm tr.deg}\,\bbk(V^*)^G+\ind\g_\xi$ with $\xi\in V^*$ generic. 
Here  ${\rm tr.deg}\,\bbk(V^*)^G=\dim  V^*\md G$ and 
$\bbk[\gt s^*]^S=\bbk[{\bH}_1,\ldots,{\bH}_\ell, F_1,\ldots,F_r]$, where $\ell=\ind\g_\xi$, 
$r= \dim  V^*\md G$, and all generators are bi-homogeneous. 
The generators $F_j$ are elements of $\gS (V)$. 
The $\gt g$-degrees of the polynomials ${\bH}_i$ are the degrees of basic invariants in 
$\gS (\g_\xi)^{G_\xi}$. 
Furthermore, $\sum_{i=1}^{\ell} \deg_V {\bH}_i + \sum_{j=1}^r \deg F_j = \dim V$, see 
\cite[Section~2]{Y} for a detailed explanation.  Thus, the only open problem is how to determine the
$V$-degrees of the ${\bH}_i$. In particular, the problem simplifies considerably, if $\ell$ is small.
\end{rmk}

%%%%%%%%%%%%%%%
\section{The classification and table}
\label{sect:rank}

\noindent
In this section, $G$ is an {\it \bfseries exceptional} algebraic group, i.e., $G$ is a simple algebraic group of
one of the types $\mathsf E_6, \mathsf E_7, \mathsf E_8, %\ (n=6,7,8), 
\mathsf F_4, \mathsf G_2$. We classify the (finite-dimensional 
rational) representations $(G:V)$ such that the symmetric invariants of $\es=\g\ltimes V$ form a polynomial ring. This will be referred to as property ({\sf FA}) for $\es$.
We also say that $\es$ (or just the action $(G:V)$)  is {\it good} (resp. {\it bad}), if ({\sf FA}) is 
(resp. is not) satisfied for $\es$.
\\ \indent To distinguish exceptional groups and Lie algebras, we write, say, $\GR{E}{7}$ for the group and
$\eus E_7$ for the respective algebra; while the corresponding  Dynkin type is referred to as $\mathsf E_7$.

\begin{ex} \label{ex:adjoint}
If $G$ is arbitrary semisimple, then $\g\ltimes\g$ always has  ({\sf FA})~\cite{takiff}.
Therefore we exclude the adjoint representations from our further consideration. 
\end{ex}

If $\bbk[\es^*]^S$ is a polynomial ring, then so is $\bbk[V^*]^G$ (Remark~\ref{rem:useful}(ii)). For this 
reason, we have to examine all representations of $G$ with polynomial rings of invariants.
%If  Table~\ref{table-ex1} includes all the representations 
If $(G:V)$ is a representation of an exceptional algebraic group such that $V\ne \g$ and $\bbk[V]^G$ is 
a polynomial ring, then $V$ or $V^*$ is contained in Table~\ref{table-ex1}. 
This can be  extracted from the classifications in \cite{ag79} or \cite{gerry1}.  Furthermore, the algebras
$\bbk[V]^G$ and $\bbk[V^*]^G$ (as well as 
$\gS(\g\ltimes V)^{G\ltimes V}$ and $\gS(\g\ltimes V^*)^{G\ltimes V^*}$) are isomorphic, so that it is 
enough to keep track of only $V$ or $V^*$.
As in~\cite{ag79,alela1,Y16}, %md-ko},
we use the Vinberg-Onishchik numbering of fundamental weights $\varpi_i$. Here $H$ is a generic isotropy 
group for $(G:V^*)$ and column ({\sf FA}) refers to the presence of property that  $\bbk[\es^*]^S$ is a 
polynomial ring, where $\es=\g\ltimes V$. We write $q(V\md G)$ for the sum of degrees of the basic invariants in $\bbk[V^*]^G$. 

In Table~\ref{table-ex1}, the group $H$ is always reductive. Since $G$ is semisimple, this implies that 
the action $(G:V^*)$ is stable in all cases,  see~\cite[Theorem~7.15]{VP}. The fact that $G$ is semisimple 
means also that  $G$, as well as $G\ltimes V$, has only trivial characters and therefore has no
proper semi-invariants.

\begin{table}[ht]
\caption{Representations of the exceptional groups with polynomial ring $\bbk[V^*]^G$}   \label{table-ex1}
\begin{center}
\begin{tabular}{>{\sf}c<{}|>{$}c<{$}>{$}c<{$}>{$}c<{$}>{$}c<{$}>{$}c<{$}>{$}c<{$}ccc}
 {\rus N0}&  G & V & \dim V & \dim V^*\md G & q(V\md G) & H & $\ind\es$ & ({\sf FA})   & Ref. \\ \hline \hline
1a & \GR{G}{2} & \varpi_1 & 7 & 1 &  2 &  \GR{A}{2} & 3 & $+$ & Eq.~\eqref{eq:chain2}\\  
1b &  & 2\varpi_1 & 14 & 3 & 6  & \GR{A}{1}  & 4 & + & \cite[Ex.\,4.8]{Y16} \\
1c &   & 3\varpi_1 & 21 & 7 & 15  & \{1\} & 7 & + &   Example~\ref{ex:h=0}\\[.5ex] \hline
2a & \GR{F}{4} & \varpi_1 & 26 & 2 & 5  & \GR{D}{4}  & 6 & $-$ &  Example~\ref{ex:F4-fi1} \\ 
2b &   & 2\varpi_1 & 52 & 8 & 22  & \GR{A}{2}  & 10 & $+$ & Eq.~\eqref{eq:chain2} \\ 
\hline
3a & \GR{E}{6}  & \varpi_1 & 27 & 1 & 3  & \GR{F}{4} & 5 & +  &  Example~\ref{ex:E6-fi1} \\  
3b &   & \varpi_1+\varpi_5 & 54 & 4 & 12  & \GR{D}{4}  &  8 & $-$ & Example~\ref{ex:E6-mnogo-fi1}\\ 
3c &   & 2\varpi_1 & 54 & 4 & 12  & \GR{D}{4}  & 8 & $-$  & Example~\ref{ex:E6-mnogo-fi1} \\
3d &   & 3\varpi_1 & 81 & 11 & 36  & \GR{A}{2}  & 13 & $+$ & Theorem~\ref{thm:E6-good}\\ 
3e &   & 2\varpi_1+\varpi_5 & 81 & 11 & 36  & \GR{A}{2}  & 13 & $+$ & Theorem~\ref{thm:E6-good} \\ 
\hline
4a & \GR{E}{7} & \varpi_1 & 56 & 1 & 4  & \GR{E}{6}  & 7 & $-$ &Theorem~\ref{thm:E7-bad}\\
4b &   & 2\varpi_1 & 112 & 7 & 28  & \GR{D}{4}  & 11 & $-$  & Example~\ref{ex:E7-mnogo-fi1}\\ \hline
%5 & \GR{E}{8} & \varpi_2\varpi'^3 & 24 & 7 & 20   & no\\  \hline
\end{tabular}
\end{center}
\end{table}

We provide below necessary explanations. % to the last column.
\begin{ex}   \label{ex:E6-fi1}
Consider item {\sf 3a} in the table. Here $V^*\md G=\mathbb A^1$, i.e.,
$\bbk[V^*]^G=\bbk[F]$ for some homogeneous $F$. The divisor $D=\{\xi\in V^*\mid F(\xi)=0\}$ contains a 
dense $G$-orbit,
%and the null-cone $\N_G(V^*)$ contains a dense $G$-orbit, 
say $G{\cdot}\eta$, whose stabiliser is the semi-direct product 
$\g_\eta=\mathfrak{so}_9\ltimes \varpi_4$. A generic isotropy group for
$(G:V^*)$ is the exceptional group $\GR{F}{4}$ and $\g_\eta$ is a $\BZ_2$-contraction of $\eus F_4$.
By~\cite[Theorem\,4.7]{coadj}, $\gS(\g_\eta)^{G_\eta}$ is a polynomial ring whose degrees 
of basic invariants are the same as those for $\eus F_{4}$. Therefore, using Theorem~\ref{V-rank-1}, we 
obtain that $\bbk[\es^*]^S$ is a polynomial ring.
\end{ex}

\begin{ex}   \label{ex:h=0}
If $\h=0$, then $\bbk[\es^*]^S\simeq \bbk[V^*]^G$~\cite[Theorem\,6.4]{p05} (cf.~\cite[Example\,3.1]{Y16}). Therefore item~{\sf 1c} is a good case.
\end{ex}
\begin{ex}   \label{ex:F4-fi1}
The semi-direct product in {\sf 2a} is a $\BZ_2$-contraction of $\eus E_6$. This is one of the four bad 
$\BZ_2$-contractions of simple Lie algebras, i.e., $\bbk[\es^*]^S$ is not a polynomial ring here, 
see~\cite[Section 6.1]{Y16}.
\end{ex}

If $G$ is semisimple, $V$ is a reducible $G$-module, say $V=V_1\oplus V_2$, %%and the action 
%%% $(G:V_1)$ is stable, 
then there is a trick that allows us to relate the polynomiality problem for the symmetric invariants of $\es=\g\ltimes V$ to a smaller semi-direct product. The precise statement is as follows.
%More precisely, if that semi-direct product do

\begin{prop}   \label{prop:trick}
With $\es=\g\ltimes (V_1\oplus V_2)$ as above,
let $H$ be a generic isotropy group for $(G:V^*_1)$ and $\gt h=\Lie H$. If\/ $\bbk[\es^*]^S$ is a polynomial ring, then so is\/
$\bbk[\tilde\q^*]^{\tilde Q}$, where $\tilde\q=\h\ltimes (V_2\vert_H)$. 
\end{prop}
\begin{proof}
Consider the (non-reductive) semi-direct product $\q=\g\ltimes V_2$. Then
$\es=\q\ltimes V_1$. It is assumed that the unipotent radical of $Q$, $1\ltimes V_2$, acts trivially on $V_1$.
%%% Then $\es\simeq \g\ltimes (V_1\oplus V_2)$,  as required.
If $\xi\in V^*_1$ is generic for $(G:V^*_1)$, then it is also generic for $(Q:V^*_1)$. If $G_\xi=H$, then the corresponding isotropy group in $Q$
is $Q_\xi=H\ltimes V_2$ and $\q_\xi=\h\ltimes V_2\simeq \tilde \q$, where $V_2$ is considered as 
$H$-module.
Since $G$ is semisimple, all hypotheses of Proposition~\ref{non-red} are satisfied for $Q$ in place of $G$ 
and $V=V_1$. Therefore $\gS(\q_\xi)^{Q_\xi} =\gS(\q_\xi)^{\q_\xi}$ 
is a polynomial ring. 
\end{proof}

\begin{rmk}
It is easily seen that the passage from $(G:V=V_1\oplus V_2)$ to $(H:V_2)$, where $H=\gig(G:V_1)$, preserves generic isotropy groups. For generic stabilisers, this appears already in \cite[\S\,3]{alela1}.
\end{rmk}
One can use Proposition~\ref{prop:trick} as a tool for proving that $\bbk[\es^*]^S$ is not a polynomial ring.

\begin{ex}   \label{ex:E6-mnogo-fi1}
In case~{\sf 3b}, we take $Q=\GR{E}{6}\ltimes \varpi_1$ and $V_1=\varpi_5$. Then 
$\es=(\eus E_6\ltimes\varpi_1)\ltimes\varpi_5\simeq \eus E_6\ltimes(\varpi_1+\varpi_5)$.
If $\xi\in V^*_1=\varpi_1$ is generic, then $\g_\xi\simeq \eus F_4$  and 
$\varpi_1\vert_{\GR{F}{4}}=\tilde\varpi_1+\odin$ \cite{alela1}. Therefore, $\q_\xi$ is isomorphic to the semi-direct product  
related to item 2a, modulo a one-dimensional centre. Therefore, $\gS(\q_\xi)^{Q_\xi}$ is not 
a polynomial ring (see Example~\ref{ex:F4-fi1}),  and we conclude, using Proposition~\ref{prop:trick}, that 
$\bbk[\es^*]^S$ is not a polynomial ring, too.
\\ \indent
In case~{\sf 3c}, we take the same $Q$ and $V_1=\varpi_1$. The rest is more or less the same, since
$\GR{F}{4}\ltimes (\varpi_1\vert_{\GR{F}{4}})$ is again a generic isotropy subgroup for $(Q:V_1^*)$. 
\end{ex}

\begin{ex}   \label{ex:E7-mnogo-fi1}
In case~{\sf 4b}, we take $Q=\GR{E}{7} \ltimes \varpi_1$ and $V_1=\varpi_1$.
Note that $\varpi_1$ is a symplectic $\GR{E}{7}$-module. If $\xi\in V_1^*$ is generic, then the corresponding stabiliser in $\eus E_7$
is  isomorphic to $\eus E_6$~\cite{H,alela1}. Thereby  
$\q_\xi=\eus E_6\ltimes (\varpi_1\vert_{\eus E_6})=\eus E_6\ltimes (\tilde\varpi_1+\tilde\varpi_5+2\odin)$ \cite{alela1}. 
Hence $\q_\xi$ represents item~{\sf 3b} (modulo a two-dimensional centre) and we have already demonstrated in the previous example that here $\bbk[\q^*_\xi]^{Q_\xi}$ is not 
a polynomial ring!
\end{ex}

The output of Examples~\ref{ex:E6-mnogo-fi1} and \ref{ex:E7-mnogo-fi1} is that there is the tree of 
reductions to a ``root'' bad semi-direct product:
\[
\xymatrix{  & (\GR{E}{6}, 2\varpi_1)\ar[dr] &   \\
(\GR{E}{7}, 2\varpi_1)\ar[r]  & (\GR{E}{6}, \varpi_1+\varpi_5)\ar[r] &  \text{\framebox{$(\GR{F}{4},\varpi_1)$}},
}\]
and therefore all these items represent bad semi-direct products.
Using Proposition~\ref{prop:trick} in 
a similar fashion, one obtains another tree  of reductions:
\beq   \label{eq:chain2}
\xymatrix{
(\GR{E}{6}, 2\varpi_1+\varpi_5)\ar[dr] & & & \\
(\GR{E}{6}, 3\varpi_1)\ar[r] & (\GR{F}{4},2\varpi_1)\ar[r] & (\GR{D}{4},\varpi_1+\varpi_3+\varpi_4)\ar[r] & (\GR{G}{2},\varpi_1). 
%\\ & & & 
}
\eeq
Some details for the passage from $\eus E_6$ to $\mathfrak{so}_8$ and then to $\eus G_2$ are given below in the proof of Theorem~\ref{thm:E6-good}.
Note that the representations occurring in~\eqref{eq:chain2} have one and the same generic isotropy group, 
namely $SL_3$. As we will shortly see, tree~\eqref{eq:chain2} consists actually of good cases. Here 
our strategy is to prove that both ``crown'' $\GR{E}{6}$-cases are good. To this end, we need some 
properties of the representation $(\GR{G}{2}, \varpi_1)$ related to the ``root'' case. 

\begin{lm}    \label{lm:G2-case}
Let $v_1$ be a highest weight vector in the $\GR{G}{2}$-module $\varpi_1$ and 
$Q:=(\GR{G}{2})_{v_1}$ the respective isotropy group. Then 
{\sf (i)} $\q=\Lie Q$ has the {\sl codim}--$2$ property and {\sf (ii)} the coadjoint representation of $Q$  
has a polynomial ring of invariants whose degrees of basic invariants are $2,3$.
\end{lm}
\begin{proof} A generic isotropy group for $(\GR{G}{2}:\varpi_1)$ is connected and isomorphic to 
$SL_3$~\cite{H} and $\q$ is a contraction of $\mathfrak{sl}_3$ (See \cite[Ch.\,7, \S\,2]{t41} for Lie algebra contractions).
We also have $\q=\el{\ltimes}\n$, where $\el=\gt{sl}_2$ and 
the nilpotent radical $\n$ is a 5-dimensional $\BZ$-graded non-abelian Lie algebra of the form 
\[
         \n(1)\oplus\n(2)\oplus\n(3)=\bbk^2_{I}\oplus \bbk\oplus \bbk^2_{II} .
\] 
Here $\bbk^2_{I}$ and $\bbk^2_{II}$ are standard $\tri$-modules and $\bbk$ is the trivial $\tri$-module.
Let $\{a_1,b_1\}$ be a basis for $\bbk^2_{I}$, $\{a_2,b_2\}$  a basis for $\bbk^2_{II}$, and $\{u\}$ a basis 
for $\bbk$. W.l.o.g., we may assume that $[a_1,b_1]=u$, 
$[a_1,u]=a_2$, and $[b_1,u]=b_2$. 

{\sf (i)} \ Since $\q$ is a contraction of $\mathfrak{sl}_3$ and $\ind \mathfrak{sl}_3=2$, we have
$\ind\q\ge 2$. On the other hand, if $0\ne \xi\in\n(3)^*\subset \n^*\subset \q^*$, then $\dim\q_\xi=2$.
Hence $\ind\q=2$ and $\n(3)^*\setminus\{0\}\subset \q^*_{\sf reg}$. Since $\dim\n(3)=2$, the last property readily implies that
$\q^*\setminus \q^*_{\sf reg}$ cannot contain divisors.

{\sf (ii)} \ It is  easily seen that 
\[
      {\mathbf h}_1=2a_1b_2-2b_1a_2+u^2 \in\gS(\n) \subset \gS(\q)
\]
is a $\q$-invariant. There is also another invariant of 
degree three. Let $\{e,h,f\}$ be a standard basis  of $\tri$ (i.e., $[e,f]=h, [h,e]=2e, [h,f]=-2f$). We assume that 
$[e,a_1]=0$ and $[f,a_1]=b_1$, which implies $[h,a_1]=a_1$ and $[h,b_1]=-b_1$. Then 
\[
      {\mathbf h}_2=b_2^2e+a_2b_2h-a_2^2f+u(a_1b_2-a_2b_1)+\frac{1}{3}u^3
\]
is an $\tri$-invariant and in addition, the following Poisson bracket can be computed:
\[
    \{a_1,{\mathbf h}_2\}=a_2b_2(-a_1)-a_2^2(-b_1)+a_2(a_1b_2-a_2b_1)-ua_2u+u^2a_2=0 .
\] 
Since $\el=\tri$ and $a_1$ generate $\q$ as Lie algebra, ${\bf h}_2$ is also a $\q$-invariant. The 
polynomials ${\mathbf h}_1$ and ${\mathbf h}_2$ are algebraically independent, because
${\mathbf h}_1\in {\gS}(\n)$ and  ${\bf h}_2\not\in {\gS}(\n)$. %The sum of their degrees is $5$. 
By {\sf (i)}, $\q$ has the {\sl codim}--$2$ property. Since $\dim\q=8$, $\ind\q=2$, and $b(\q)=5$,
we have $\deg {\mathbf h}_1+\deg {\mathbf h}_2=b(\q)$.
Therefore, ${\bf h}_1$ and ${\bf h}_2$ freely generate ${\gS}(\q)^{Q}$, see $(\blacklozenge)$ in
Section~\ref{sect:prelim}.
\end{proof}

\begin{rmk}
In this case, $(\GR{G}{2})_{v_1}$ is a so-called {\it truncated parabolic subgroup}. 
Symmetric invariants of truncated (bi)parabolics 
were intensively studied by Fauquant-Millet and Joseph, see e.g. 
\cite{FMJ, J}. 
Let $\rr_{\sf tr}$ be a truncated (bi)parabolic in type {\sf A} or {\sf C}. Then 
${\eus S}(\rr_{\sf tr})^{\rr_{\sf tr}}$ is a polynomial ring in $\ind\rr_{\sf tr}$ homogeneous generators and 
the sum of their degrees is equal to $b(\rr_{\sf tr})$. 
The same properties hold for many truncated (bi)parabolics in other types~\cite{J}. It is very probable that 
a sufficient condition of~\cite{J} is satisfied for $(\eus G_2)_{v_1}$. However, we prefer to keep the explicit construction of generators.   
\end{rmk}

\begin{thm}   \label{thm:E6-good}
If\/ $\es =\eus E_6\ltimes (3\varpi_1)$ or $\es =\eus E_6\ltimes(2\varpi_1+\varpi_5)$, then $\bbk[\es^*]^S$ is a polynomial ring.
\end{thm}
\begin{proof}
Here $G=\GR{E}{6}$, $\g=\eus E_6$, and $V=3\varpi_1$ or $2\varpi_1+\varpi_5$. 
Accordingly, $V^*=3\varpi_5$ or $2\varpi_5+\varpi_1$. In both cases, $\ind\es=13$, $V^*\md G\simeq \mathbb A^{11}$, and a 
generic isotropy group for $(G:V^*)$ is $SL_3$. By \cite[Theorem\,2.8 and Lemma\,3.5({\sf i})]{Y16}, there 
are bi-homogeneous irreducible polynomials $F_1,F_2\in \bbk[\es^*]^S$ such that $\deg_{\g}F_1=2$, $\deg_{\g}F_2=3$, and 
their restrictions to $\g^*\times\{\xi\}$, where $\xi\in V^*$ is $G$-generic,  are the basic invariants of
$SL_3$. Here, for $\g^*\times\{\xi\} \subset \es^*$, we use the  isomorphism
\beq    \label{eq:canon-iso}
  \bbk[\g^*\times\{\xi\}]^{G_\xi\times V}\simeq \gS(\g_\xi)^{G_\xi} ,
\eeq
see \cite[Lemma~2.5]{Y16}. %%%Proposition~\ref{non-red}.
By \cite[Lemma\,3.5({\sf ii})]{Y16}, $F_1$ and $F_2$ generate $\bbk[\es^*]^S$ over $\bbk[V^*]^G$ if and only if
the restrictions $F_1\vert_{\g^*\times D}$ and $F_2\vert_{\g^*\times D}$ are algebraically 
independent for any $G$-stable homogeneous divisor $D\subset V^*$. Let us prove that this is really the case. 
\\  \textbullet\quad %\indent
If there is a non-trivial relation for the restrictions of $F_1$ and $F_2$ to $\g^*\times D$, then this also yields 
a non-trivial relation for the subsequent restrictions of $F_1$ and $F_2$ to $\g^*\times\{\eta\}$, where 
$\eta$ is $G$-generic in $D$.
\\  \textbullet\quad 
If $G_\eta\simeq SL_3$, i.e., $\eta$ is also $G$-generic in $V^*$, then $F_1\vert_{\g^*\times\{\eta\}}$ and $F_2\vert_{\g^*\times \{\eta\}}$
are nonzero elements of $\gS(\mathfrak{sl}_3)^{SL_3}$ of degrees $2$ and $3$, respectively.
The invariants of degree $2$ and $3$ in $\gS(\mathfrak{sl}_3)^{SL_3}$ are uniquely determined, up to a scalar factor, and they are algebraically independent. Hence $F_1\vert_{\g^*\times D}$ and $F_2\vert_{\g^*\times D}$ are algebraically independent for such divisors.
\\  \textbullet\quad
However, it can happen that a divisor $D$ contains no ``globally'' $G$-generic points. To circumvent this difficulty, consider three projections from $V^*$ to its simple constituents, and their (non-trivial!) restrictions to $D$:
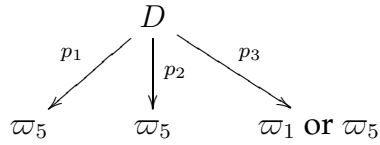
\[ \xymatrix{
                  & D\ar[dl]_{p_1}\ar[d]^{p_2}\ar[dr]^{p_3} & \\
         \varpi_5 & \varpi_5 &  \varpi_1 \text{ or } \varpi_5  
}\]
For $\eta=y_1+y_2+y_3\in D\subset V^*$, we have
$p_i(\eta)=y_i$. Since $D$ is a divisor, at least two of the $p_i$'s are dominant.  
Without loss of generality, we may assume that $p_1$ and $p_2$ are dominant, whereas
$p_3$ is not and then $\ov{p_3(D)}$ is a divisor in $p_3(V^*)$. (For, if all $p_i$'s are dominant, then $D$ contains a globally generic point.)
Therefore, we can take $y_1,y_2$ to be generic in the $p_i(V^*)$.
Then $y_1+y_2$ is $G$-generic in $2\varpi_5$, $G_{y_1+y_2}=Spin_8$, and
$\varpi_1\vert_{Spin_8}\simeq \varpi_5\vert_{Spin_8}\simeq \tvp_1+\tvp_3+\tvp_4+3\odin$.
Here $G_\eta=(Spin_8)_{y_3}$.
Having obtained a $Spin_8$-stable divisor $\ov{p_3(D)}$ in the $Spin_8$-module $ \tvp_1+\tvp_3+\tvp_4+3\odin$, we consider $\tilde V:=\tvp_1+\tvp_3+\tvp_4$ and
$\tilde D:=\ov{p_3(D)}\cap \tilde V$. If $\tilde D=
\tilde V$, then $(Spin_8)_{y_3}=SL_3$ for some $y_3$, i.e., again $G_\eta=SL_3$.
\\  \textbullet\quad
If $\tilde D$ is a divisor in $\tilde V$, then we
can play the same game with $Spin_8$ and $\tilde D$. Let $y_3=x_1+x_3+x_4\in \tilde D$, where
$x_i\in\tvp_i$. Again, at least two of the projections $\tilde p_i:\tilde V\to \tvp_i$ ($i=1,3,4$) are dominant.
Without loss of generality, we may assume that $\tilde p_1$ and $\tilde p_3$ are dominant and then
$x_1$ and $x_3$ are generic elements. Then
$(Spin_8)_{x_1+x_3}\simeq \GR{G}{2}$, $(\GR{G}{2})_{x_4}=(Spin_8)_{y_3}$,
and $\tvp_4\vert_{\GR{G}{2}}\simeq \hat \varpi_1+\odin$.
The structure of $\GR{G}{2}$-orbits in $\hat \varpi_1$ shows that either $x_4$ is $\GR{G}{2}$-generic and then
$(\GR{G}{2})_{x_4}\simeq SL_3$ or $x_4$ is a highest weight vector and 
$(\GR{G}{2})_{x_4}\simeq (\GR{G}{2})_{v_1}$, cf. Lemma~~\ref{lm:G2-case}.

Thus, for any $G$-stable divisor $D\subset V^*$, there is $\eta\in D$ such that 
$G_\eta=SL_3$ or $(\GR{G}{2})_{v_1}$, and in both cases $\gS(\g_\eta)^{G_\eta}$ is generated by  
algebraically independent invariants of degree $2$ and $3$ (see Lemma~\ref{lm:G2-case} for the latter).
It follows that $F_1\vert_{\g^*\times D}$ and $F_2\vert_{\g^*\times D}$ are algebraically 
independent, and we are done.
\end{proof}

Combining Proposition~\ref{prop:trick}, Theorem~\ref{thm:E6-good}, and tree~\eqref{eq:chain2}, we conclude that
cases {\sf 3d, 3e, 2b}, and {\sf 1a} are good. (Note also that we have found one good case related to a representation of the classical Lie algebra $\eus D_4$.)
Thus, it remains to handle only the semi-direct product  
$\eus E_7\ltimes \varpi_1$ (case {\sf 4a}).

\begin{thm}   \label{thm:E7-bad}
If\/ $\es =\eus E_7\ltimes \varpi_1$, then $\bbk[\es^*]^S$ is not a polynomial ring.
\end{thm}
\begin{proof}
Here  $G=\GR{E}{7}$, $\g=\eus E_7$, and $\bbk[V^*]^G=\bbk[F]$. By~\cite{H}, a generic isotropy group for
$(G:V^*)$ is 
connected and isomorphic to $\GR{E}{6}$, and  the null-cone
$D=\{\xi\in V^*\mid F(\xi)=0\}$ contains  a dense $G$-orbit, say $G{\cdot}y$, whose isotropy group 
$G_y$ is connected and isomorphic to $\GR{F}{4}\ltimes\varpi_1$.

Assume that $\bbk[\es^*]^S$ is polynomial. 
By \cite[Theorem\,2.8 and Lemma\,3.5({\sf i})]{Y16}, since $\ind\es=7$ and $\gig$ is  
$\GR{E}{6}$, there are bi-homo\-ge\-ne\-ous polynomials $H_1,\dots,H_6$ 
such that $\bbk[\es^*]^S=\bbk[F,{\bH}_1,\dots,{\bH}_6]$ and the ${\bH}_i\vert_{\g^*\times\{\xi\}}$'s yield the basic invariants of $\eus E_6$ for a generic $\xi$. Here we again use Proposition~\ref{non-red} and isomorphism~\eqref{eq:canon-iso}.  
By \cite[Lemma~3.5(ii)]{Y16}, the restrictions ${\bH}_i\vert_{\g^*\times D}$, $i=1,\dots,6$ remain algebraically independent.
On the other hand, let us consider further restrictions ${\bH}_i\vert_{\g^*\times \{y'\}}$, $i=1,\dots,6$, where $y'$ belongs to the dense $G$-orbit in $D$. Recall that $\g_{y'}\simeq \eus F_{4}\ltimes \varpi_1$ and the latter is a bad $\BZ_2$-contraction of $\eus E_{6}$, see Example~\ref{ex:F4-fi1}. Moreover, the algebra of symmetric invariants of $\eus F_{4}\ltimes \varpi_1$ does not have algebraically independent invariants whose degrees are the same as the degrees of basic invariants of $\eus E_{6}$~\cite[Section~6.1]{Y16}. This implies that ${\bH}_i\vert_{\g^*\times \{y'\}}$, $i=1,\dots,6$ must be algebraically dependant for any $y'\in G{\cdot}y$.  

Let $L({\bH}_1\vert_{\g^*\times \{y'\}}, \dots, {\bH}_6\vert_{\g^*\times \{y'\}})=0$ be a polynomial relation 
for {\bf some} $y'$. Since the ${\bH}_i$'s are $G$-invariant, the relation with the same coefficients holds 
for {\bf all} $y'\in G{\cdot}y$. Hence, this dependance can be lifted to $\g^*\times G{\cdot}y$ and then 
carried over to $\g^*\times D$. This contradiction shows that the ring $\bbk[\es^*]^S$ cannot be polynomial.
\end{proof}

Summarising, we obtain the main result of the article:

\begin{thm} \label{thm:main}
Let $G$ be an exceptional algebraic group, $V$ a (finite-dimensional rational) $G$-module, and
$\es=\g\ltimes V$. Then $\bbk[\es^*]^S$ is a polynomial ring if and only if one of the following two conditions is satisfied: {\sf (i)} $V=\g$; {\sf (ii)}  $V$ or $V^*$ represents one of the seven good cases in Table~\ref{table-ex1}. 
\\  Moreover, by Ra\"is' formula (see Section~\ref{sect:kos-th}), one has 
$\trdeg \bbk[\es^*]^S=\ind\es=\dim V^*\!\md G +\rk H$.
\end{thm}

\noindent
In the good cases, we neither construct generators nor give their degrees. The reason is that the main results of \cite{Y16} as well as Theorem~\ref{V-rank-1} are purely existence theorems. Yet, as explained in Remark~\ref{rem:deg-V}, a great deal of information on the degrees is available. If $\ell$ is small, then using ad hoc methods one can determine all the degrees.
Let us see how this can be done for an example with $\ell=2$. 
Take item~{\sf 1a}. Since $\gt h=\gt{sl}_3$, we have 
$\bbk[\gt s^*]^S=\bbk[V^*]^G[\bH_1,\bH_2]$ with $\deg_{\g}\bH_1=2$, $\deg_{\g} \bH_2=3$. 
Set $a_i=\deg_V \bH_i$. It can easily be seen that $\es$ has the {\sl codim}--$2$ property. 
The discussion in Remark~\ref{rem:deg-V} shows that $a_1+a_2=5$. 
Following the same strategy as in \cite[Prop.\,3.10]{Y16}, one can produce an $S$-invariant 
of bi-degree $(2,2)$. This implies that $a_1=2$ and $a_2=3$.

\end{document}